\documentclass[11pt]{article}
\usepackage{amsmath}
\usepackage{amssymb}
\usepackage{algorithmic}
\usepackage{algorithm}
\usepackage{pgfplots}
\usepackage{caption}

\usepackage{color}
\usepackage{tikz}
\usetikzlibrary{arrows,
	arrows.meta,
	bending,calc,patterns,angles,quotes}
\usepackage{float}

\newtheorem{theorem}{Theorem}

\newtheorem{remark}{Remark}

\newtheorem{corollary}{Corollary}
\def\endpf{{\ \hfill\hbox{\vrule width1.0ex height1.0ex}\parfillskip 0pt
	}}
	\newenvironment{proof}{\noindent{\bf Proof:}}{\endpf}
\begin{document}	
\title{Steady-state optimization of an  exhaustive L\'evy storage process with intermittent output and random output rate}
\author{Royi Jacobovic\thanks{Department of Statistics and Data Science; The Hebrew University of Jerusalem; Jerusalem 9190501; Israel.
		{\tt royi.jacobovic@mail.huji.ac.il/offer.kella@gmail.com}}\and Offer Kella\footnotemark[1] \thanks{Supported by grant No. 1647/17 from the Israel Science Foundation and
		the Vigevani Chair in Statistics.}}
\maketitle
\begin{abstract}
	Consider a regenerative storage process with a nondecreasing L\'evy input (subordinator) such that every cycle may be split into two periods. In the first (off) the output is shut off and the workload accumulates. This continues until some stopping time. In the second (on), the process evolves like a subordinator minus a positive drift (output rate) until it hits the origin. In addition, we assume that the output rate of every on period is a random variable which is determined at the beginning of this period. For example, at each period, the output rate may depend on the workload level at the beginning of the corresponding busy period. We derive the Laplace-Stieltjes transform of the steady state distribution of the workload process and then apply this result to solve a steady-state cost minimization problem with holding, setup and output capacity costs. It is shown that the optimal output rate is a nondecreasing deterministic function of the workload level at the beginning of the corresponding on period. 
\end{abstract}


\bigskip
\noindent {\bf AMS Subject Classification (MSC2010):} 60K25, 90B22, 90C15.

\section{Introduction}
Motivated by M/G/1 queues with server vacations, \cite{kella1998} considers a regenerative storage process with a nondecreasing L\'evy input (subordinator). In particular, the author assumes that every cycle has two periods. In the first period (off) the output is shut off and hence the workload accumulates. This continues until some stopping time. Then, at the second period (on) the process evolves like a subordinator minus a positive drift (output rate) until it hits the origin. This paper is about a related model in which the output rate is a random variable which changes between cycles and is determined at the beginning of every on period. For example, the output rate may be a function of the workload level or the number of customers at the beginning of every on period or may be dependent (or independent) on any other information accumulated during the preceding off period. Two possible examples for this model are the following:
\begin{enumerate}
	\item A machine when once its workload is depleted it undergoes some maintenance and therefore experiences a down time. During this down time, potential work flows in but is not processed until this down time ends and therefore accumulates in storage. Information about the maintenance process and the arrival of new jobs is gathered during this down time, at the end of which the operator needs to decide at which speed to set the machine. After this, it is left to its own devices until it signals that its buffer is empty again. At this time the procedure is repeated.
	
	\item A water dam in which water accumulates. When water reaches some level, which remains fixed,  an operator must set a release rate which remains fixed until the water reaches some lower level. At that time, this process repeats. The release rate in each cycle can randomly depend on the history of rainfalls and water level since the previous time when the lower level was reached.    
\end{enumerate}
 In what follows, we derive the Laplace-Stieltjes transform (LST) of the steady-state workload of such a system. Then,  a steady-state cost minimization problem is studied, where, it is assumed that the workload level at the beginning of every on period is observed. In addition, there are constant holding, setup and output capacity cost rate which create a non-trivial trade-off which has to be balanced by choosing an output rate for the upcoming on period. Once the output rate is chosen, it cannot be changed until the end of the on period. It is shown that there exists an optimal output rate policy which is a nondecreasing deterministic function of the initial workload level. In addition, a two phase approach is applied to solve this optimization problem. It is shown that this approach is valid even when the objective functional is not convex. Furthermore, we show that this approach is applicable to a similar problem where the only available information is the number of customers in the system at the beginning of each on period. It is assumed that the service demands of these customers are unknown and hence we refer to this problem as an optimization with partial information. The rest of this work is organized as follows: Section \ref{sec: general model} is about on finding the steady state LST of the workload process. In section \ref{sec: optimization} we present the general statement of the cost-minimization problem along with a two-phase approach for solving it. Section \ref{sec: examples} contains two queueing examples including numerical illustrations. In Section \ref{sec: partial information} we show that the problem with partial information can be similarly solved.

\subsection{Related literature}
	 A summary of recent developments regarding L\'evy queues is given by \cite{debicki2015} and a more general treatment of L\'evy processes with some applications to queueing theory and other fields is provided by \cite{kypryanou2006}. In addition, recent surveys regarding queueing models with server vacations are given by e.g.  \cite{ke2010,tian2006}.  Related models discuss L\'evy queues or more generally L\'evy driven storage process with an exponent which is adapted over time with respect to some given mechanism. For example, \cite{bekker2009} discusses a L\'evy queue with hysteretic control. \cite{bekker2008} provides an analysis of a L\'evy driven model such that whenever the workload crosses some fixed level, a timer with a constant length is activated and when it expires the exponent is changed, unless the origin had already been hit before. \cite{bekker2009a} discusses a L\'evy queue with feedback information. In addition, this work considers the case where the exponent might be changed at Poisson arrivals. Another work about a L\'evy queue which evolves between exponential timers is \cite{palmowski2011}. In addition \cite{asmussen2000} discusses the case where the exponent is controlled by an external Markovian environment. A special case of the current model is when the output rate is determined as a function of the length of the previous vacation. This makes the current paper related to \cite{boxma2008,boxma2010} which consider the opposite case, i.e., when the vacations lengths depend on the evolution of the workload process during the previous busy period.
	 
\section{Model setup}\label{sec: general model}
Consider a probability space $(\Omega,\mathcal{F},P)$ with a filtration $\mathbb{F}=(\mathcal{F}_t)_{t\geq0}$  satisfying the usual conditions (right-continuous and augmented).
With respect to this filtration, let $J\equiv(J_t)_{t\geq0}$ be a finite-mean subordinator. Specifically, it is a nondecreasing right-continuous process which starts at the origin and has stationary independent increments. It is well known that $Ee^{-\alpha J_t}=e^{-\eta(\alpha)t}$ for $\alpha,t\geq0$, where  

\begin{equation}
\eta(\alpha)=c \alpha+\int_{(0,\infty)}\left(1-e^{-\alpha x}\right)\nu({\rm d}x)\ \ , \ \ \forall\alpha\geq0\,,
\end{equation}
$c\geq0$ and $\nu$ is the associated L\'evy measure satisfying $\int_{(0,\infty)}x\wedge1\nu({\rm d}x)<\infty$. Denote 
\begin{equation}
\rho\equiv EJ_1=\eta'(0)=c+\int_{(0,\infty)}x\nu({\rm d}x)=c+\int_0^\infty\nu\left((x,\infty)\right){\rm d}x
\end{equation}
and consider a regenerative storage system with the following evolution during the first cycle. The input of the storage system is $J$ which is described above. At time $t=0$ the system is off, i.e., there is no output at that time and the workload accumulates. This continues until some stopping time $\tau$ for which $J(\tau)>0$ with probability one and $E\tau<\infty$. In addition, note that, $\tau$ is a stopping time with respect to the filtration $\mathbb{F}$ which may be richer than the filtration which is generated by $J$. Now, at epoch $\tau$ a random rate $R$ is chosen. It is assumed that $R>\rho$ with probability one. Then, the workload process $\{J_t-R(t-\tau);t\geq\tau\}$ continues until $\tau+T^\tau$ which is the first time that it hits the origin. We denote this process by $W^\tau\equiv W^\tau_t$. In particular, the process $W^\tau$ restricted to off periods is the same as the L\'evy-driven clearing process $Z^\tau\equiv Z^\tau_t$ which was discussed in Section 3 of \cite{kella1998}. In addition, the process $W^\tau$ restricted to on periods is denoted by $Y^\tau\equiv Y_t^\tau$. It is possible to apply the results of Sections 3 and 4 of \cite{kella1998} in order to derive the LST of the steady state distribution of $W^\tau$. Let $\tilde{Z}^\tau$ and $\tilde{Y}^\tau$ be the steady state distributions of $Z^\tau$ and $Y^\tau$ respectively.  Thus, we shall deduce that the LST of the steady state workload process $\tilde{W}^\tau$ is given by 

\begin{equation} \label{eq: gen model LST}
\tilde{W}^\tau(\alpha)=\frac{E\tau}{E\tau+E\frac{J_\tau}{R-\rho}}\tilde{Z}^\tau(\alpha)+\frac{E\frac{J_\tau}{R-\rho}}{E\tau+E\frac{J_\tau}{R-\rho}}\tilde{Y}^\tau(\alpha)\ \ , \ \ \forall\alpha>0
\end{equation}
where $S(\alpha)=Ee^{-\alpha S}$ is a notation for the LST of a nonnegative random variable $S$ at $\alpha>0$. Now, the LST of $\tilde{Z}^\tau$ is given by Theorem 3.2 of \cite{kella1998} so it is left to figure the LST of $\tilde{Y}^\tau$. To this end, denote $V=J_\tau$ and whenever $E\frac{V}{R-\rho}<\infty$, let
\begin{equation}
P_{V,R}(A)\equiv \frac{E\frac{V}{R-\rho}1_A}{E\frac{V}{R-\rho}} \ \ , \ \ \forall A\in\mathcal{F}\,. 
\end{equation}
We denote the expectation with respect to $P_{V,R}(A)$ by $E_{V,R}$, i.e., 
\begin{equation}
E_{V,R}S=\frac{E\left(\frac{V}{R-\rho}\cdot S\right)}{E\left(\frac{V}{R-\rho}\right)}
\end{equation}
for every random variable $S$ for which this expectation exists. 

\begin{figure}[H] 
	\centering
	\textbf{Figure 1}
	\begin{tikzpicture}
	\draw[->] (-2,0)--(-2,2) node [anchor=south] {$\text{Workload}$};
	\draw[->] (-2,0)--(7.5,0);
	
	\draw (-2,0)--(-2,0) node [anchor=north] {$0$};
	\draw (7.5,0)--(7.5,0) node [anchor=north] {$\text{Time}$};
	\draw[line width=0.5mm] (-2,0)--(-1,0);
	\draw[dashed, line width=0.5mm] (-1,0)--(-1,0.75);
	\draw[line width=0.5mm] (-1,0.75)--(0,0.75);
	\draw[dashed,line width=0.5mm] (0,0.75)--(0,1.25);
	\draw[line width=0.5mm] (0,1.25)--(0.75,1.25);
	\draw[line width=0.5mm] (0.75,1.25)--(2,0);
	\draw (0.5,0.7)--(0.5,0.7) node [anchor=north] {$1$};

	\draw[line width=0.5mm] (2,0)--(2.5,0);
	\draw[dashed, line width=0.5mm] (2.5,0)--(2.5,0.5);
	\draw[line width=0.5mm] (2.5,0.5)--(3,0.5);
	\draw[dashed, line width=0.5mm] (3,0.5) -- (3,1);
	
	\draw[line width=0.5mm] (3,1)--(4,0.5);
	\draw[dashed, line width=0.5mm] (4,0.5)--(4,1);
	\draw[line width=0.5mm] (4,1)--(5.5,0.25);
	\draw[dashed, line width=0.5mm] (5.5,0.25)--(5.5,0.75);
	\draw[line width=0.5mm] (5.5,0.75)--(7,0); 
	\draw[line width=0.5mm] (7,0)--(7.48,0);
	\draw[line width=0.5mm] (4.35,0.7) -- (4.35,0.7) node [anchor=north] {$2$};
	
	\end{tikzpicture}
	\caption*{Sample path during the first two cycles. Note: output rates during the first and second cycles are different.} 
	\label{fig: path}
\end{figure}
\begin{theorem} \label{thm: LST-formula}
	Let  $U,N,e_1,e_2,\ldots$ be independent random variables such that $U\sim U[0,1]$ and $\{e_n;n\geq1\}$ is an i.i.d sequence such that $e_1(\alpha)=\frac{\eta(\alpha)}{\rho\alpha}$ for every $\alpha\geq0$. In addition, $U,e_1,e_2,\ldots$ are independent of $(V,R)$ and $\left(N+1\right)|(V,R)\sim \text{Geometric}\ (1-\frac{\rho}{R})$. Then,   
	\begin{align}
	\tilde{Y}^{\tau}(\alpha)&=\frac{E\left[\frac{1-e^{-\alpha V}}{\alpha R-\eta(\alpha)}\right]}{E\left(\frac{V}{R-\rho}\right)}\label{eq: LST1}\\&=E_{V,R}  \exp{\left\{-\alpha\left(UV+\sum_{i=1}^Ne_i\right)\right\}}\label{eq: LST2}\ \ , \ \ \forall \alpha\geq0\,.
	\end{align}
\end{theorem}

\begin{proof}
	Theorem 2 of \cite{kella1992} implies that given $(V,R)$ 
	\begin{equation*}
	\left[\alpha R-\eta(\alpha)\right]\int_0^te^{-\alpha Y^\tau _s}{\rm d}s+e^{-\alpha V}-1\ \ , \ \ \forall t\geq0
	\end{equation*}
	is a zero-mean martingale. Therefore, by the optional sampling theorem with the stopping time $T^\tau\wedge t$ we have upon taking $t\uparrow\infty$ (monotone convergence) 
	\begin{equation*}
	E\left[\int_0^{T^\tau}e^{-\alpha Y^\tau _s}{\rm d}s\bigg|V,R\right]=\frac{1-e^{-\alpha V}}{\alpha R-\eta(\alpha)}\ \ , \ \ P\text{-a.s.}
	\end{equation*}  
	In addition, it is known (see Equation 2.10 of \cite{kella1998}) that  
	\begin{equation*}
	E\left[T^\tau\bigg|(V,R)\right]=\frac{V}{R-\rho}\ \ , \ \ P\text{-a.s.}
	\end{equation*} 
	Thus, conditioning and un-conditioning with respect to $(V,R)$ imply \eqref{eq: LST1}. To prove \eqref{eq: LST2}, observe that for every $\alpha>0$
	\begin{align}
	\tilde{Y}^{\tau}(\alpha)\nonumber&=\frac{E\left[\frac{1-e^{-\alpha V}}{\alpha R-\eta(\alpha)}\right]}{E\left(\frac{V}{R-\rho}\right)}\\&=\frac{E\left[\frac{1-e^{-\alpha V}}{\alpha V}\cdot\frac{\alpha\left(R-\rho\right)}{\alpha R-\eta(\alpha)}\cdot\frac{V}{R-\rho}\right]}{E\left(\frac{V}{R-\rho}\right)}\\&=E_{V,R}\left(\frac{1-e^{-\alpha V}}{\alpha V}\cdot\frac{1-\frac{\rho}{R}}{1-\frac{\rho}{R}\cdot\frac{\eta(\alpha)}{\rho\alpha}}\right)\nonumber\\&=E_{V,R}\left[e^{-\alpha UV}\sum_{n=0}^\infty\left(1-\frac{\rho}{R}\right)\left(\frac{\rho}{R}\right)^n\left(Ee^{-\alpha e_1}\right)^n\right]\nonumber\\&=E_{V,R}e^{-\alpha\left(UV+\sum_{i=1}^Ne_i\right)}\nonumber
	\end{align}
	and the proof is complete.
\end{proof}
\\ \\
The following is immediate. 

\begin{corollary}\label{cor: expectation formula}
	Denote 
	\begin{equation} \label{eq: mu}
	\mu\equiv Ee_1=-\frac{\eta''(0)}{2\rho}=\frac{\int_{(0,\infty)}x^2\nu({\rm d}x)}{2\rho}\,.
	\end{equation}
	 Then, 
	\begin{equation}
	E\tilde{Y}^\tau=E_{V,R}\left(UV+\sum_{n=1}^Ne_i\right)=\frac{E\left[\frac{V^2}{2(R-\rho)}+\mu\rho\frac{V}{(R-\rho)^2}\right]}{E\left(\frac{V}{R-\rho}\right)}
	\end{equation}
\end{corollary}
We note that the case where $R=\rho+sV$ with $s>0$, Theorem \ref{thm: LST-formula} implies that 
\begin{equation} \label{eq: inspection paradox}
E\tilde{Y}^\tau=\frac{EV}{2}+\frac{\mu\rho}{s}EV^{-1}\xrightarrow[s\to\infty]{}\frac{EV}{2}\leq\frac{EV^2}{2EV}\,.
\end{equation}
The right side of \eqref{eq: inspection paradox} would be the limit if instead of taking $R=\rho+sV$ we took $R=\rho+sEV$. We also note that in both cases $E\frac{V}{R-\rho}=\frac{1}{s}$.

\section{Optimization} \label{sec: optimization}
Recall the model introduced by  Section \ref{sec: general model} and note that $V=J(\tau)$. In addition, the following are model assumptions:
\begin{enumerate}
	\item $EV^2<\infty$. 
	\item A constant holding cost $h\in(0,\infty)$ per unit of workload per unit of time. 
	\item A constant setup cost  $K\in(0,\infty)$ at the beginning of every busy period.
	\item A constant output capacity cost rate $d>0$. That is, if a rate $r$ is chosen for $t$ units of time, then the cost of the operator is $dtr$. This is to model the cost involved with higher consumption rate of resources as output rate is increased.	
	\item There exists a constant $\rho<r<\infty$ such that $\rho<R\leq r$, $P$-a.s. $r$ models the maximal output rate as a consequence of technological constraint. As mentioned in Remark \ref{remark: r=infinity}, it is required for our analysis to go through. 
\end{enumerate}
Our goal is to minimize the steady state expected cost, where the minimization is over all possible distributions of the random variable $R$. The general expression for the expected cost is defined as follows
\begin{align} \label{eq: C(R) definition}
C(R)&\equiv\frac{E\left[K+h\left(\int_0^\tau J_tdt+\int_\tau^{\tau+T^\tau}\left[J_t-R(t-\tau)\right] {\rm d}t\right)+dT^\tau R\right]}{E\left(\tau+T^\tau\right)}\nonumber\\&=\frac{K+h\left(E\tau E\tilde{Z}^\tau+ET^\tau E\tilde{Y}^\tau\right)+dE\left[RE\left(T^\tau|R,V\right)\right]}{E\tau+ET^\tau}\nonumber\\&=\frac{hE\tau E\tilde{Z}^\tau+K+dEV+d\rho E\frac{V}{R-\rho}+hE\left[\frac{V^2}{2(R-\rho)}+\mu\rho\frac{V}{(R-\rho)^2}\right]}{E\tau+E\frac{V}{R-\rho}}\,.
\end{align}
Note that the last equality is justified by Corollary \ref{cor: expectation formula} and the identity $E\left(T^\tau|R,V\right)=\frac{V}{R-\rho}$, $P$-a.s. To proceed, consider the change of variable
\begin{equation}\label{eq: change of variable}
X=V\left(\frac{1}{R-\rho}-\frac{1}{r-\rho}\right)
\end{equation}   
and let 
\begin{align}
K_1&\equiv hE\tau E\tilde{Z}^\tau+K+\left[d+\frac{d\rho}{r-\rho}+\frac{h\mu\rho}{(r-\rho)^2}\right]EV+\frac{h}{2(r-\rho)}EV^2\,,\nonumber\\K_2&\equiv d\rho+2\frac{h\mu\rho}{r-\rho}\,,\\K_3&\equiv E\tau+\frac{EV}{r-\rho}\,.\nonumber
\end{align}
Then, since $EV^2<\infty$ an equivalent optimization problem is:

\begin{equation} \label{optimization: general}
\begin{aligned}
& \min:
& & \frac{K_1+K_2EX+hE\left(\frac{V}{2}X+\frac{\mu\rho}{V} X^2\right)}{K_3+EX} \\
& \text{ s.t.}
& & X\in\mathcal{F} \\ & & & 0\leq X, \ \ P\text{-a.s.}
\end{aligned}
\end{equation}
where $X\in\mathcal{F}$ means that $X$ is a random variable. \\ \\ We solve this problem in two phases: 

\subsection*{Phase I}
Fix some  $\alpha>0$ and solve the problem under the additional constraint $EX=\alpha$. That is,

\begin{equation} \label{optimization: phase1}
\begin{aligned}
& {\min:}
& &  E\left[\frac{VX}{2}+\mu\rho\frac{X^2}{V}\right] \\
& \text{ s.t.}
& & X\in\mathcal{F}\\ &  & & 0\leq X,\ \ P\text{-a.s.} \\  & & & EX=\alpha\,. 
\end{aligned}
\end{equation}
Let $a^+=\max\{a,0\}$ and $a\wedge b=\min\{a,b\}$ for every $a,b\in\mathbb{R}$. In addition, for every $v>0$ and $\lambda\in\mathbb{R}$, define 

\begin{align}\label{eq: x definition}
&x(v,\lambda)=\frac{v\left(\lambda-\frac{v}{2}\right)^+}{2\mu\rho}=\frac{v\left(\lambda-\frac{v}{2}\right)}{2\mu\rho}1_{\left(\frac{v}{2},\infty\right)}(\lambda)\nonumber\,,\\& \xi(\lambda)=Ex(V,\lambda)\,,\\&f(\alpha)=\frac{1}{4\mu\rho}EV\left[\lambda_\alpha^2-\left(\frac{V}{2}\right)^2\right]^+ \nonumber
\end{align}
where $\lambda_\alpha$ is defined by the next theorem.
\begin{theorem} \label{thm: phase I}
	$\xi(\cdot)$ is continuous on $[0,\infty)$, there exists $\lambda_\alpha>0$ such that $\xi(\lambda_\alpha)=\alpha$ and $X=x(V,\lambda_\alpha)$ is a solution of \eqref{optimization: phase1}. $f(\alpha)$ is the optimal value of \eqref{optimization: phase1} which is finite, convex, non-decreasing and continuous on $[0,\infty)$.	
\end{theorem}

We note that from \eqref{eq: change of variable} and \eqref{eq: x definition} it follows that for every $\alpha,v\in[0,\infty)$ the optimal output rate 

\begin{equation}
R=R(v,\alpha)=\rho+\left[\frac{1}{r-\rho}+\frac{\left(\lambda_\alpha-\frac{v}{2}\right)^+}{2\mu\rho}\right]^{-1}
\end{equation}
which is nondecreasing in $v$. In particular this holds for the optimal $\alpha$ resulting from Subsection \ref{subsec: phase II}. We also note that since $x(v,\lambda)$ is nondecreasing in $\lambda$, then $\lambda_\alpha$ is nondecreasing in $\alpha$ implying that $R(v,\alpha)$ is nonincreasing in $\alpha$.\newline
  
\begin{proof}
	Note that $\xi(0)=0$. Let $\lambda\geq0$ and notice that 
	\begin{equation} \label{eq: h upper bound}
	x(V,\lambda)=\frac{V\left(\lambda-\frac{V}{2}\right)^+}{2\mu\rho}\leq\frac{\lambda^2}{\mu\rho}<\infty\,.
	\end{equation}
	Thus, dominated convergence implies that $\xi(\cdot)$ is continuous and finite on $[0,\infty)$. In addition, monotone convergence implies that $\xi(\lambda)\rightarrow\infty$ as $\lambda\to\infty$ and hence there exists $\lambda_\alpha>0$ such that $\xi(\lambda_\alpha)=\alpha$.
	
	The optimization \eqref{optimization: phase1} is a special case of optimizing quadratic function with random coefficients. That is, 
	\begin{equation} \label{optimization: qudratic}
	\begin{aligned}
	& \min:
	& & E\left[AX^2+BX+C\right] \\
	& \text{ s.t.}
	& & X\in\mathcal{F}\\ & & & 0\leq X,\ \ P\text{-a.s.} \\  & & & EX=\alpha 
	\end{aligned}
	\end{equation}
	with $A=\frac{\mu\rho}{V}$, $B=\frac{V}{2}$ and $C=0$.
	Thus, the fact that $X=x(V,\lambda_\alpha)$ is an optimum follows directly from Corollary 1 of \cite{Jacobovic2019} (see also Section 4 there). Note that $f(\alpha)$ resulting by a substitution of $x\left(V,\lambda_\alpha\right)$ into the objective function of \eqref{optimization: qudratic} with the appropriate $A,B$ and $C$. Thus,
	\begin{equation}
		f(\alpha)=\frac{1}{4\mu\rho}EV\left[\lambda_\alpha^2-\left(\frac{V}{2}\right)^2\right]^+\leq\frac{\lambda_\alpha}{\mu\rho}<\infty\,. 
		\end{equation}
	Moreover, observe that for $X\geq0$, $P$-a.s., $0\leq\frac{VX}{2}+\mu\rho\frac{X^2}{V}$, $P$-a.s. 
	Therefore, all other results regarding $f(\cdot)$ follows directly by Proposition 1 of \cite{Jacobovic2019}.   
\end{proof}

\begin{remark}\label{remark: bounded rate}
		\normalfont If there exists a minimal output rate, that is $\rho<r_{\min}\leq R\ ,\ P\text{-a.s.}$, then it can be checked that $X\leq r_0V,\ , \ P\text{-a.s.}$ where $r_0=\left(r_{\min}-\rho\right)^{-1}-(r-\rho)^{-1}>0$. Thus, it is enough to consider $\alpha\in(0,r_0EV]$. In addition, the results of Section 4 of \cite{Jacobovic2019} can be applied directly to show that there exists $\lambda_\alpha\geq0$ such that $Ex(V,\lambda_\alpha)\wedge(r_0V)=\alpha$ and $x(V,\lambda_\alpha)\wedge(r_0V)$ is an optimum. In addition, Proposition 1 of \cite{Jacobovic2019} can be used to show that $f(\cdot)$ satisfies the same properties stated in Theorem \ref{thm: phase I}.    
\end{remark}

\begin{remark}\label{remark: r=infinity}
		\normalfont 
		The proof of Theorem \ref{thm: phase I} doesn't require the assumption $EV^2<\infty$. In fact, it requires only that $V$ is such that $P\left(V>0\right)>0$. Therefore, if $d=0$ and $r=\infty$, then  and hence the results of Theorem \ref{thm: phase I} remain valid with weaker assumptions on the distribution of $V$. Especially, note that $r=\infty$ models a situation where at the end of an off period the server can either choose any finite output rate  and starts an on period with this rate or to pick an infinite output rate, i.e. to perform a clearing and then to go back on a new off period.
\end{remark}

\subsection*{Phase II} \label{subsec: phase II}
Recall that for every $\alpha\geq0$ the optimal value $f(\alpha)$ of phase I is finite and $f(0)=0$. The second phase is the optimization problem
\begin{equation} \label{optimization: phase 2}
\begin{aligned}
& \min:
& & \frac{K_1+K_2\alpha+hf(\alpha)}{K_3+\alpha} \\
& \text{ s.t.}
& & \alpha\geq0
\end{aligned}
\end{equation}
In addition, for every $v\geq0$ and $\lambda\geq0$, let
	\begin{align}
	&a(\lambda,v)\equiv v\left(\lambda-\frac{v}{2}\right)^+\,,\\& b(\lambda,v)\equiv v\left[\lambda^2-\left(\frac{v}{2}\right)^2\right]^+\,,\nonumber\\&A(\lambda)\equiv Ea(\lambda,V)\,,\nonumber\\&B(\lambda)=Eb(\lambda,V)\,,\nonumber\\&
	\tilde{a}(\lambda)\equiv EV1_{[0,2\lambda]}(V)\,.\nonumber
	\end{align}
and define 
	\begin{equation}\label{eq: g definition}
	G(\lambda)\equiv\frac{K_1+K_2\xi(\lambda)+hf\left[\xi(\lambda)\right]}{K_3+\xi(\lambda)}=\frac{K_1+\frac{K_2}{2\mu\rho}A(\lambda)+\frac{h}{4\mu\rho}B(\lambda)}{K_3+\frac{A(\lambda)}{2\mu\rho}}\,.
	\end{equation} 
Since $\xi(\cdot)$ is a nondecreasing continuous function such that $\xi(0)=0$ and $\xi(\lambda)\to\infty$ as $\lambda\to\infty$, then,
in order to solve \eqref{optimization: general}, it suffices to minimize $G(\cdot)$ on $[0,\infty)$. The following theorem guarantees that there exists a minimizer of $G(\cdot)$ which can be derived by standard line-search techniques. 
  
\begin{theorem} \label{thm: upper bound} 
Let $\lambda^*\equiv\frac{\left(K_1-K_2K_3\right)^+}{K_3h}$. Then, $G(\cdot)$ is absolutely continuous unimodal function on $[0,\infty)$ with a minimizer in $\left[0,\lambda^*\right]$. 
\end{theorem}
\begin{proof}
For every $a,b\geq0$ define 

\begin{equation}
U(a,b)\equiv\frac{K_1+\frac{K_2}{2\mu\rho}a+\frac{h}{4\mu\rho}b}{K_3+\frac{a}{2\mu\rho}}
\end{equation}
and notice that 
\begin{equation}
G(\lambda)=U\left[A(\lambda),B(\lambda)\right]\ \ , \ \ \lambda\geq0\,.
\end{equation}	
In addition, observe that $a(\cdot)$ and $b(\cdot)$ are convex with respect to $\lambda$ and hence so are $A(\cdot)$ and $B(\cdot)$. Now, since $A(\cdot)$ and $B(\cdot)$ are convex, they are both functions of bounded variation. Therefore, since $U(\cdot)$ is differentiable on $[0,\infty)^2$, then Theorem 33 of page 81 in \cite{Protter2005} implies that $G(\cdot)$ has a density which equals to 
\begin{equation}\label{eq: chain rule}
g(\lambda)=U_a\left[A(\lambda),B(\lambda)\right]D^+A(\lambda)+U_b\left[A(\lambda),B(\lambda)\right]D^+B(\lambda)\ \ , \ \ \forall\lambda\geq0
\end{equation}
where $U_a,U_b$ are the partial derivatives of $U$ with respect to $a$ and $b$ and $D^+A,D^+B$ are the right-derivatives of $A(\cdot)$ and $B(\cdot)$ with respect to $\lambda(\cdot)$.
Note that from this it follows that $G(\cdot)$ is absolutely continuous on $\left[0,\infty\right)$.
Now, for every $v\geq0,\lambda\geq0$ the right derivatives of $a(\cdot)$ and $b(\cdot)$ with respect to $\lambda$ are given by
	\begin{align}
	&D^+a(\lambda,v)=v1_{[0,2\lambda]}(v)\,,\\& D^+b(\lambda,v)=2\lambda v1_{[0,2\lambda]}(v)\,.\nonumber
	\end{align}
	Therefore, recalling that $EV^2<\infty$, hence $EV<\infty$ and we have by dominated convergence that, for $\lambda>0$, 
	\begin{equation}
	\tilde{a}(\lambda)= D^+A(\lambda)=\frac{D^+B(\lambda)}{2\lambda}\,.
	\end{equation}
	Therefore, for every $\lambda>0$, $g(\lambda)$ is given by
	\begin{equation}\label{eq: g derivative}
	g(\lambda)=\frac{\frac{\tilde{a}(\lambda)}{2\mu\rho}\left[K_3K_2-K_1+\lambda K_3h+\frac{h}{4\mu\rho}\left[2\lambda A(\lambda)+B(\lambda)\right]\right]}{\left[K_3+\frac{A(\lambda)}{2\mu\rho}\right]^2}\,.
	\end{equation}
	Now, observe that $\tilde{a}(\lambda)$  and $2\lambda A(\lambda)+B(\lambda)$ are nonnegative and nondecreasing on $\left[0,\infty\right)$. Therefore, the numerator of $g(\cdot)$ can change its sign at most once.  Clearly, if $K_1\leq K_2K_3$, then $G(\cdot)$ is nondecreasing and nonnegative  in which zero is a minimizer. On the other hand, if $K_1>K_2K_3$, since the numerator converges to infinity as $\lambda\to\infty$, then there exists a minimizer which is bounded from above by $\lambda^*$, the solution of 
	\begin{equation}
	K_3K_2-K_1+\lambda K_3h=0
	\end{equation} 
	and the result follows.   
\end{proof}

\begin{remark} 
	\normalfont From \eqref{eq: g derivative}, it can be seen that for every $\lambda\in\left(0,\lambda^*\right)$, $g(\lambda)$ is bounded by 
	\begin{equation*}
	M\equiv\frac{\tilde{a}\left(\lambda^*\right)}{2K_3^2\mu\rho}\left[K_3K_2+K_1+\lambda^* K_3h+\frac{h}{4\mu\rho}\left[2\lambda^* A\left(\lambda^*\right)+B\left(\lambda^*\right)\right]\right]
	\end{equation*}
	 and hence $G(\cdot)$ is Lipschitz on $\left[0,\lambda^*\right]$ with a constant $M$. This can be useful for the purpose of error bounds. 
	\end{remark}
	\begin{remark}
		\normalfont It is easy to see that when the distribution of $V$ has no atoms then $G(\cdot)$ is continuously differentiable with $G'=g$ (see \eqref{eq: g derivative}). 
	\end{remark}
\section{Two queueing examples}\label{sec: examples}
\subsubsection*{Example 1:}
Consider an $M/M/1$ system such that at the beginning of every busy period the server observes the service demand of the first customer and determines a service rate for the rest of this cycle in order to minimize \eqref{optimization: general}. Let $\nu\in(0,\infty)$ be the arrival rate and assume that service demands are exponentially distributed with rate $\theta\in(0,\infty)$. Thus, in such a case, $V\sim\exp(\theta)$, $\tau\sim\exp(\nu)$, $\tilde{Z_t}=0,\forall t\geq0$,
$\rho=\frac{\nu}{\theta}$ and $\mu=\frac{1}{2}\left(1+\frac{\text{Var}(V)}{EV}\right)=\frac{1+\theta^{-1}}{2}$. For simplicity, assuming $r=\rho+1$, the resulting parameters of the model are
\begin{align}
K_1&= K+\frac{d}{\theta}\left(1+\frac{\nu}{\theta}\right)+\frac{h}{\theta^2}\left(\nu+2\right)\,,\\K_2&= \frac{1}{\theta}\left(d\nu+\frac{2h\nu}{\theta}\right)\,,\nonumber\\K_3&= \frac{1}{\nu}+\frac{1}{\theta}\,.\nonumber
\end{align}
In particular, since $V\sim\exp(\theta)$ then for every $\lambda\geq0$
\begin{align}
&A(\lambda)=\frac{\lambda}{\theta}\psi\left(2\lambda,2\right)-\frac{1}{\theta^2}\psi\left(2\lambda,3\right)\,,\\&B(\lambda)=\frac{\lambda^2}{\theta}\psi\left(2\lambda,2\right)-\frac{3}{2\theta^3}\psi\left(2\lambda,4\right)\nonumber
\end{align}
where $\psi(\cdot,n)$ is the cumulative distribution function of Erlang distribution with shape $n$ and rate $\theta$.  The following Figure 2 illustrates the results of Theorem \ref{thm: upper bound} regarding the graph of $G(\cdot)$ as a function of $\lambda$ when $\theta=\nu=1$. For example, the left plot presents a negative effect of setup cost on the optimal output rate. Namely, we consider $K=h=d=1$ (blue), $K=75,h=d=1$ (red) and $K=200,h=d=1$ (green). The central plot presents a positive effect of holding cost on the optimal output rate. That is, we consider $K=200,h=d=1$ (blue), $K=200,h=5,d=1$ (red) and $K=200,h=30,d=1$ (green). The right plot presents a negative effect of output capacity cost on the optimal output $g(\cdot)$. That is, we consider $K=h=d=1$ (blue), $K=1=h,d=30$ (red) and $K=h=1,d=100$. (green).
	
\begin{figure}[H]
	\centering
	\textbf{Figure 2}	
	 
	\includegraphics[scale=0.7]{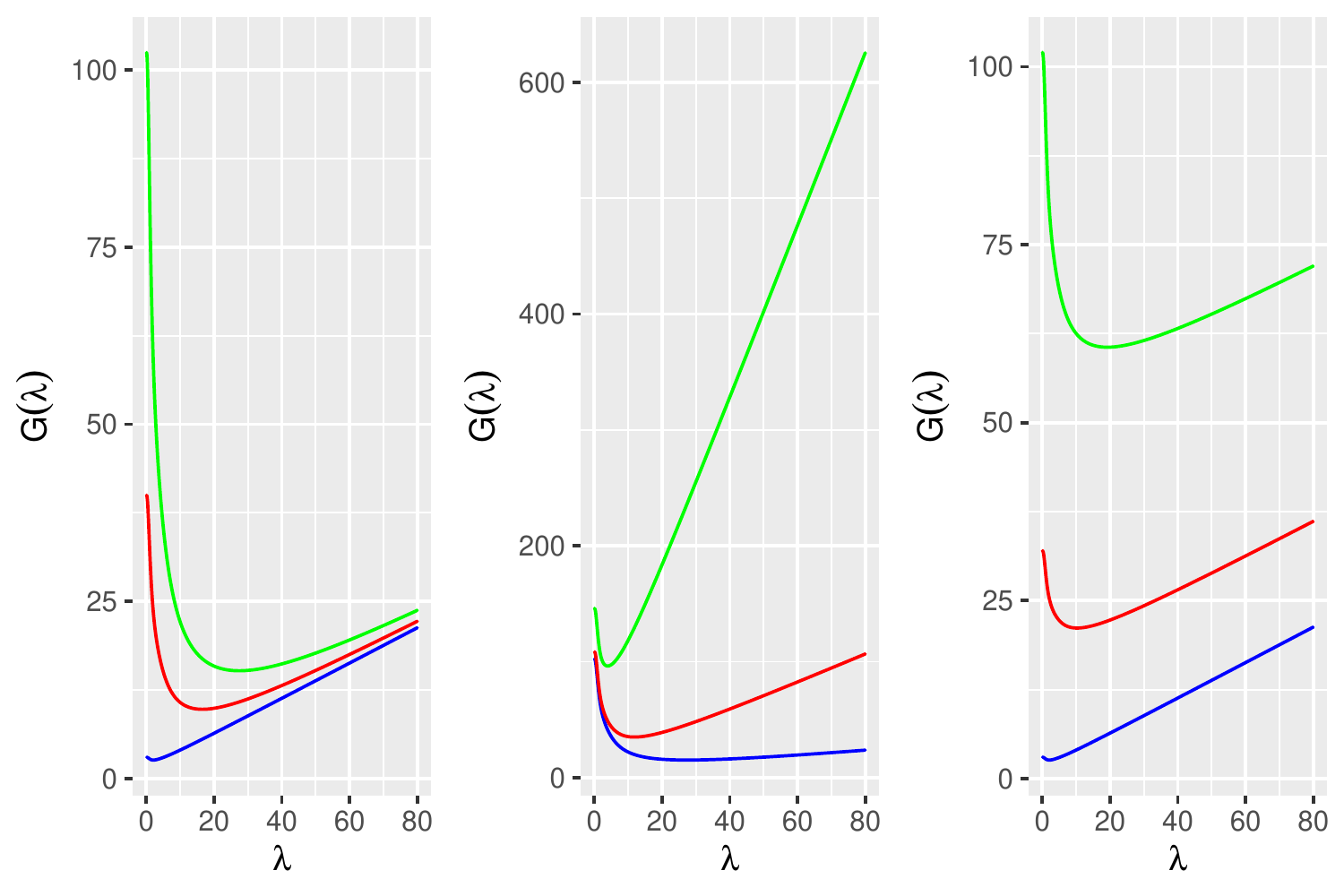}
		\label{fig: exponential}  
\end{figure}

\subsubsection*{Example 2:}
 Consider an $M/G/1$ system with a server who determines the service rate at the beginning of every busy period in order to minimize \eqref{optimization: general}. Let $\nu=\frac{1}{2}$ be the arrival rate and assume that service demands are uniformly distributed on $\left[0,1\right]$. Thus, in such a case, $V\sim U\left[0,1\right]$, $\tau\sim\exp\left(\frac{1}{2}\right)$, $\tilde{Z_t}=0,\forall t\geq0$, $\rho=1$ and $\mu=\frac{1}{2}\left(1+\frac{\text{Var}(V)}{EV}\right)=\frac{7}{12}$. In addition, for simplicity, we assume that $r=\rho+1$. Thus, 
 \begin{align}
 K_1&= K+d+\frac{5h}{8}\,,\\K_2&=d+\frac{7h}{6}\,,\nonumber\\K_3&=2.5\,.\nonumber
 \end{align}
 
 Since $V\sim U[0,1]$, then for every $0<\lambda\leq\frac{1}{2}$ we have that $V|\{V\leq2\lambda\}\sim U[0,2\lambda]$. Therefore,  
 \begin{equation}
 A(\lambda)=P\left(V\leq2\lambda\right)E\left[V\left(\lambda-\frac{V}{2}\right)\bigg|V\leq2\lambda\right]=\frac{2\lambda^3}{3}
 \end{equation}
 and similarly $B(\lambda)=\lambda^4$. In addition, for every $\lambda>\frac{1}{2}$, 
 \begin{equation}
 A(\lambda)=EV\left(\lambda-\frac{V}{2}\right)=\frac{\lambda}{2}-\frac{1}{6}
 \end{equation}
 and 
 \begin{equation}
 B(\lambda)=E\left(\lambda^2V-\frac{V^3}{4}\right)=\frac{\lambda^2}{2}-\frac{1}{16}\,.
 \end{equation}
 An insertion of these results into the definition of $g(\cdot)$ implies that

 \[ G(\lambda)=\begin{cases} 
 \frac{K_1+\frac{K_2}{3\mu\rho}\lambda^3+\frac{h}{4\mu\rho}\lambda^4}{K_3+\frac{\lambda^3}{3\mu\rho}} & 0\leq\lambda\leq\frac{1}{2} \\ \\
 \frac{K_1+\frac{K_2}{2\mu\rho}\left(\frac{\lambda}{2}-\frac{1}{6}\right)+\frac{h}{4\mu\rho}\left(\frac{\lambda^2}{2}-\frac{1}{16}\right)}{K_3+(2\mu\rho)^{-1}\left(\frac{\lambda}{2}-\frac{1}{6}\right)} & \frac{1}{2}<\lambda<\infty
 \end{cases}\,.
 \]
 The following Figure 3 illustrates the graph of $G(\cdot)$ as a function of $\lambda$ when the costs are varying. The left plot presents a negative effect of setup cost on the optimal output rate. Namely, we consider $K=h=d=1$ (blue), $K=75,h=d=1$ (red) and $K=200,h=d=1$ (green). The central plot presents a positive effect of holding cost on the optimal output rate. That is, we consider $K=200,h=d=1$ (blue), $K=200,h=5,d=1$ (red) and $K=200,h=10,d=1$ (green). The right plot presents a positive effect of output capacity cost rate on the optimal output $G(\cdot)$. That is, we consider $K=200,h=d=1$ (blue), $K=200,h=1,d=50$ (red) and $K=200,h=1,d=100$ (green). Interestingly, unlike the case presented by Figure 2 , Figure 3 presents a case when the output capacity cost rate $d$ has a positive effect on the optimal output rate. 
 
 \begin{figure}[H] 
 	\centering
 	\textbf{Figure 3} 
 	
 	\includegraphics[scale=0.7]{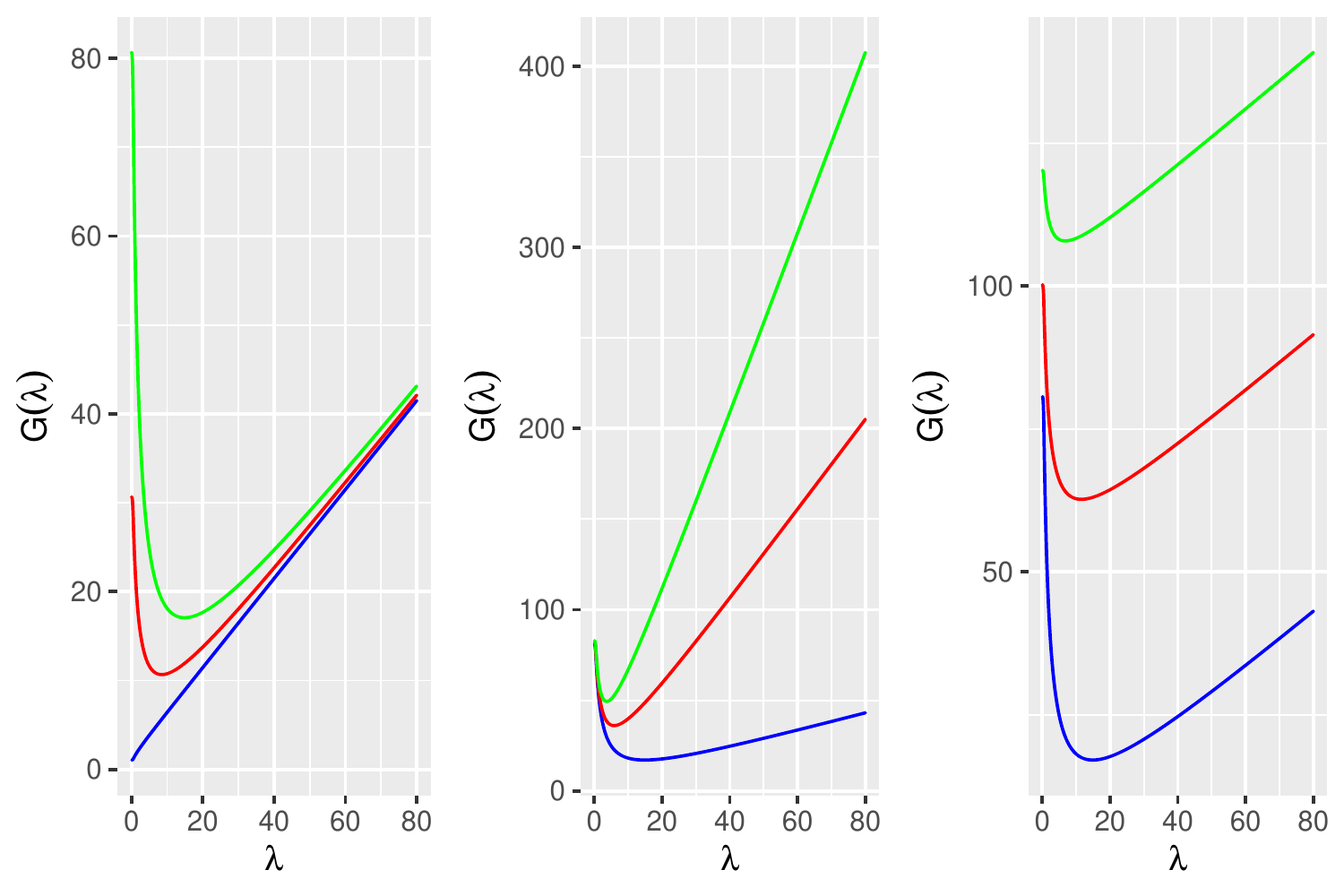}
 	\label{fig: uniform}   
 \end{figure}
 The following Figure 4, presents the same graphs which were sketched in Figure 3 on the interval $[0,1.3]$. Especially, notice that all the graphs are smooth at $0.5$. In addition, Figure 4 demonstrates the fact that the graph of $G(\cdot)$ can be neither convex nor concave. 
 
 \begin{figure}[H] 
 	\centering
 	\textbf{Figure 4}
 	
 	\includegraphics[scale=0.7]{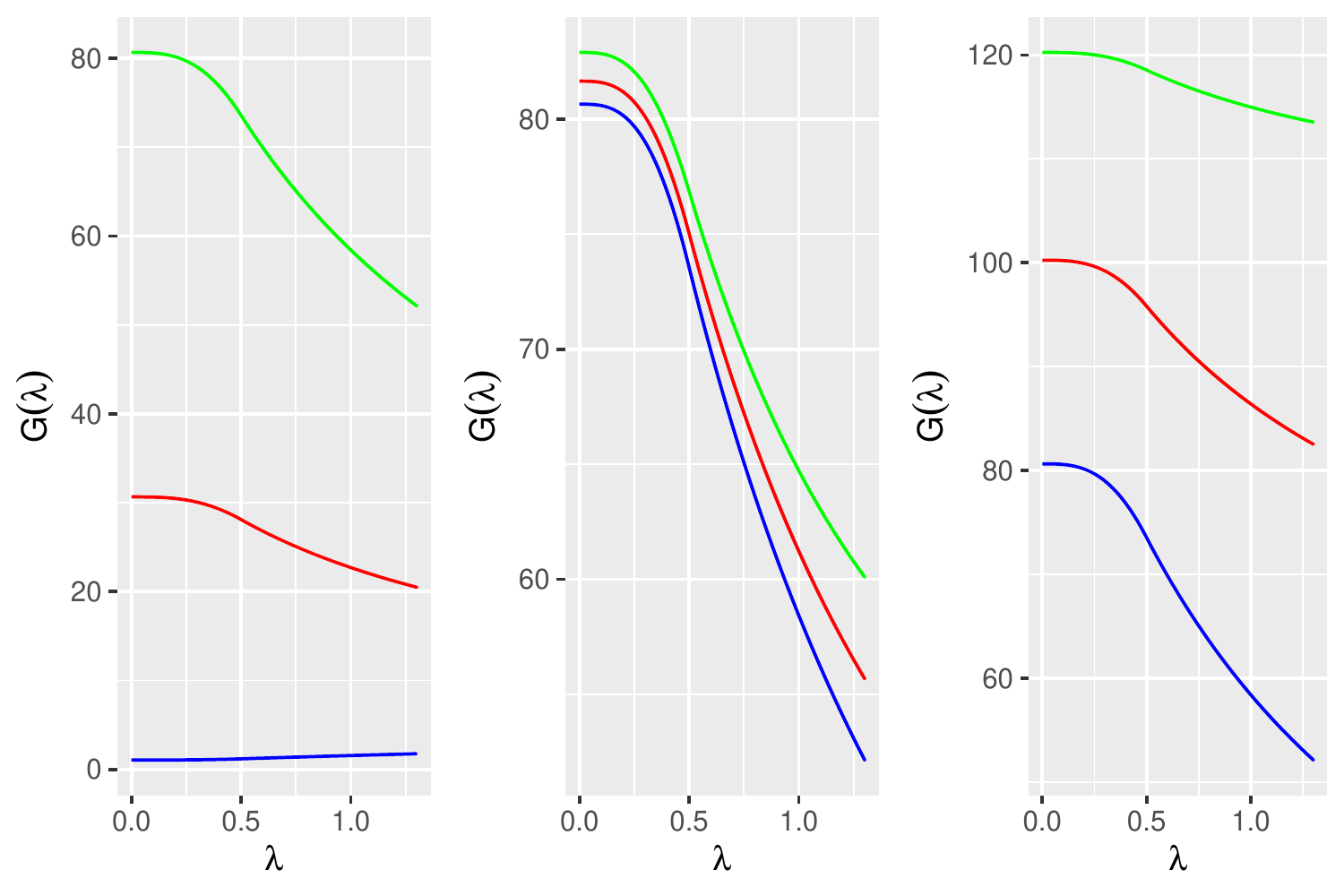}
 	\label{fig: uniform}   
 \end{figure}
 
\section{Optimization with partial information}\label{sec: partial information}
So far the assumption was that the decision maker knows the exact value of $V$ when she has to pick the output rate. To see a slightly different point of view regarding this model, consider the case where $V=\sum_{k=1}^N\tilde{V}_i$ such that $\tilde{V}_1,\tilde{V}_2,\ldots$ are i.i.d non-negative random variables and $N$ is some discrete random variable which is independent of $\{\tilde{V}_i\}_{i=1}^\infty$. In particular, we assume that $E\tilde{V_1}=\delta$ and $\text{Var}(\tilde{V}_1)=\sigma^2$. In the context of queueing systems, $N$ is interpreted as the number of customers in the system when the service rate is chosen. Now, in this part the assumption is that only $N$ is observed (and not the service times) when the output rate is picked. That is, we assume that $(N,R)$ and $\tilde{V}_1,\tilde{V}_2,\ldots$ are independent. Observe that the laws of total expectation and variance imply that 
\begin{align}\label{eq: total expectation}
\frac{1}{2}&E\frac{V^2}{R-\rho}+\mu\rho E\frac{V}{\left(R-\rho\right)^2}\nonumber\\&=\frac{1}{2}\left\{\text{Var}\left(\frac{V}{\sqrt{R-\rho}}\right)+\left(E\frac{V}{\sqrt{R-\rho}}\right)^2\right\}\\&+\mu\rho\delta E\frac{N}{\left(R-\rho\right)^2}\nonumber\\&=\frac{\sigma^2}{2}E\frac{N}{R-\rho}+\frac{\delta^2}{2}E\frac{N^2}{R-\rho}+\mu\rho\delta E\frac{N}{\left(R-\rho\right)^2}\,.\nonumber
\end{align}
Thus, an insertion of \eqref{eq: total expectation} in \eqref{eq: C(R) definition} provides an expression with numerator
\begin{align}
&hE\tau E\tilde{Z}^\tau+K+d\delta EN+\\&\left(d\rho\delta+\frac{h\sigma^2}{2}\right) E\frac{N}{R-\rho}+h\left(\frac{\delta^2}{2}E\frac{N^2}{R-\rho}+\mu\rho\delta E\frac{N}{(R-\rho)^2}\right)\nonumber
\end{align}
and denominator $E\tau+\delta E\frac{N}{R-\rho}$. 
Thus, an appropriate change of variable implies that this minimization is analogous to a special case of \eqref{optimization: general} with  $V$ which has a discrete distribution with a support which is a subset of $\left\{1,2,\ldots\right\}$.

\subsection{When $V$ has a discrete distribution} \label{subsec: discrete}
As explained in the previous part of this section, now $V$ models the number of customers in the system at the beginning of an on period. Therefore, we assume that $V$ is a random variable which receives values on the set $\left\{1,2,\ldots\right\}$ with probabilities
\begin{equation}
p_i\equiv P\left(V=i\right)\ \ , \ \ \forall i=1,2,\ldots
\end{equation}
such that $\sum_{i=1}^\infty p_i=1$. In such a case, for every $\lambda\geq0$  
\begin{equation}
A(\lambda)=\sum_{i=1}^\infty p_ii\left(\lambda-\frac{i}{2}\right)^+=\sum_{i=1}^{\lfloor2\lambda\rfloor} p_ii\left(\lambda-\frac{i}{2}\right)
\end{equation} 
and similarly
\begin{equation}
B(\lambda)=\sum_{i=1}^{\lfloor2\lambda\rfloor}p_ii\left(\lambda^2-\frac{i^2}{4}\right)\,.
\end{equation}
Therefore, for every $\lambda\geq0$ 
\begin{equation}
G(\lambda)=\frac{K_1+\sum_{i=1}^{\lfloor2\lambda\rfloor}p_i\left[\frac{K_2}{2\mu\rho}i\left(\lambda-\frac{i}{2}\right)+\frac{hi}{4\mu\rho}\left(\lambda^2-\frac{i^2}{4}\right)\right]}{K_3+(2\mu\rho)^{-1}\sum_{i=1}^{\lfloor2\lambda\rfloor}p_ii\left(\lambda-\frac{i}{2}\right)}\,.
\end{equation}
Now, by Theorem \ref{thm: upper bound} an efficient line search method can be used in order to find $i\in\left\{\frac{1}{2},\frac{2}{2},\frac{3}{2},\ldots\right\}$ such that the optimal $\lambda$ belongs to $I=[i-\frac{1}{2},i)$. Then, for every $\lambda\in I$,  
\begin{equation}
G(\lambda)=\frac{S_i\lambda^2+T_i\lambda+U_i}{Q_i+W_i\lambda}\equiv G_i(\lambda)
\end{equation}
where 
\begin{align}
&S_i=\frac{h}{4\mu\rho}\sum_{j=1}^{2i-1}p_jj\,,\nonumber\\& T_i=\frac{K_2}{2\mu\rho}\sum_{j=1}^{2i-1}p_jj\,,\nonumber\\&U_i=K_1-\sum_{j=1}^{2i-1}j^2\left(\frac{K_2}{4\mu\rho}+\frac{hj}{16\mu\rho}\right)\,,\\&Q_i=K_3-\frac{1}{4\mu\rho}\sum_{j=1}^{2i-1}p_jj^2\,,\nonumber\\& W_i=(2\mu\rho)^{-1}\sum_{j=1}^{2i-1}p_jj\nonumber
\end{align}
with the convention that an empty sum equals zero. Finally, it is left to minimize $G_i(\cdot)$ on $I$. This can be done analytically and hence, unlike the general case, when only partial information is available, the optimal output rate can be calculated efficiently with no error at all. 

\newpage
\end{document}